\documentclass[12pt]{amsart}
\usepackage{amsfonts,amsthm}
\usepackage[usenames,dvipsnames]{color}

\newtheorem{theorem}{Theorem}[section]
\newtheorem{proposition}[theorem]{Proposition}
\newtheorem{corollary}[theorem]{Corollary}
\newtheorem{remark}[theorem]{Remark}
\newtheorem{lemma}[theorem]{Lemma}
\newtheorem{definition}[theorem]{Definition}

\newcommand{\eps}{{\varepsilon}}
\newcommand{\cB}{{\mathcal{B}}}
\newcommand{\cL}{{\mathcal{L}}}
\newcommand{\cN}{{\mathcal{N}}}

\newcommand{\brrho}{{\bar{\rho}}}
\newcommand{\hsum}{\widehat\sum}

\title[Regularity of  measures for piecewise expanding maps]
{Regularity of absolutely continuous invariant measures for piecewise expanding unimodal maps.}
\author{Fabi\'an Contreras, Dmitry Dolgopyat}

\begin{document}

\maketitle

\begin{abstract}
Let $f : [0, 1] \to [0, 1]$  be a piecewise expanding unimodal map of class $C^{k+1}$, with $k \geq 1$, and $\mu = \rho dx$ the (unique) SRB measure associated to it. We study the regularity of $\rho$. In particular, points $\cN$ where $\rho$ is not differentiable has zero Hausdorff dimension, but is uncountable if the critical orbit of $f$ is dense. This improves on a work of Szewc (1984). We also obtain results about higher orders of differentiability of $\rho$ in the sense of Whitney.
\end{abstract}

\tableofcontents

\section{Introduction}
An important discovery of the 20th century mathematics is that many deterministic systems exhibit stochastic
behavior. The stochasticity is caused by exponential divergence of nearby trajectories. This instability causes many important objects associated to dynamical systems, such as attractors and invariant measures, to be fractal.

Piecewise expanding maps of the interval are among the simplest and most studied examples of chaotic
systems. They admit absolutely continuous invariant measures (a.c.i.m) \cite{LY} which enjoy exponential decay of correlations, the Central Limit Theorem for H\"older observables, and at least one of them is ergodic
(see e.g. \cite{Bal2, V}).

In this paper, we consider a class of simplest piecewise expanding
maps, so called piecewise expanding unimodal maps (PEUMs)\footnote{The precise definition of PEUMs is given at the beginning of Section \ref{Preliminary}.} of the unit interval.
PEUMs are piecewise expanding maps with only
two branches.
We study regularity of the density of a.c.i.m for PEUMs.
A classical result of A. Lasota and J. Yorke \cite{LY} says that the density, which we
denote by $\rho,$ is of bounded variation. Recall that a bounded variation function is differentiable almost everywhere (See e.g., \cite{FitzRoyden}, Corollary 6.6).
Therefore the set of non-differentiability of $\rho$ is a natural fractal set associated to our PEUM.
Let us describe the previous results about the differentiability.
In the smooth case, R. Sacksteder \cite{Sacksteder}  and K. Krzyzewski \cite{Krzyzewski} proved that when a map $f$ is expanding of class $C^k$, with $k \geq 1$, then $\rho$ is of class $C^{k-1}$.\;Later, B. Szewc \cite{Szewc} showed that if $f$ is a piecewise expanding continuous map of class $C^{k+1}$ with finitely many critical points (those points where the derivative of $f$ is not defined), with $k \geq 1$, then a density function will belong to the space
$$\{ \phi \in BV[0,1] : \phi \in C^{k} \mbox{ in } [0,1]\backslash B\},$$ where $B$ is the union of the closures of the critical orbits.
In this paper, we improve on $k=1$ case of the Szewc's theorem for PEUMs by showing that the set where $\rho$ is differentiable is larger.
Namely, we need to discard not all points in the closure of the critical orbit, but only points which are
approached by the critical orbits exponentially fast. We  also
obtain a partial converse, by showing that
 if $x$ is approached exponentially fast
by the critical orbit \emph{and  the exponent is sufficiently large} then
$\rho$ is not differentiable at $x.$

We also show that  a similar improvement is possible for $k>1$ if we consider
smoothness in the sense of Whitney, that is, we study the points where
the density admits a Taylor expansion of  order $k.$
(Of course Szewc's result is optimal for classical smoothness
since the set where the density is not differentiable is dense in $B$).
This leads to the question of describing the Taylor coefficients of the density.
Here we make use of the recent result of V. Baladi \cite{Bal1}
\footnote{Baladi was motivated by the question of regularity of invariant measure with respect to
parameters raised in the work of D. Ruelle \cite{Ru2, R1, R2}. Applications of  Baladi's result
to Ruelle's question are described in \cite{Bal1, BS1, BS2}. Our results also have applications to
the regularity question as will be detailed elsewhere.}
saying that the density $\rho$ belongs to
the set
$$BV_1=\{\phi \in BV[0,1]: \mbox{ there exists } \psi \in BV[0,1] \mbox{ s.t. } \phi'=\psi \mbox{ almost everywhere } \}.$$

In other words, the derivative of $\rho$ coincides with a function of bounded variation almost everywhere.
Accordingly, we can differentiate that function almost everywhere and call the result the second derivative of
$\rho.$ We then show that this procedure can be continued recursively and that the resulting functions indeed
provide the Taylor coefficients of $\rho.$

More precisely, the main results
of our paper can be summarized as follows.
Let $f$ be a PEUM such that both branches of $f$ are $C^{k+1}, $ with $k \geq 1$.

\begin{theorem}\label{Derivativesofrho} There is a sequence of functions $\rho_0, \rho_1,\dots, \rho_k\in BV$ such that $\rho_0=\rho$ and for $j<k, \rho_j'=\rho_{j+1}$ almost everywhere.
\end{theorem}

\begin{theorem}\label{ThW}

    \begin{itemize}
        \item[(A)] The set of points where $\rho$ is non differentiable has Hausdorff dimension zero.
        \item[(B)] If the critical orbit is dense then the set of points where $\rho$ is non differentiable is uncountable.
        \item[(C)] There is a set $\cN$ such that $\mathcal{HD}(\cN)=0$ and $\rho$ is $k$ differentiable in the sense of Whitney on $[0,1]-\cN$. That is, if $\bar{x}\not\in\cN$ then
            $$ \rho(x)-\rho(\bar{x})=\sum_{m=1}^k \frac{\rho_m(\bar{x})}{m!} (x-\bar{x})^m+
            o\left(\left(x-\bar{x}\right)^{k}\right). $$
    \end{itemize}

\end{theorem}

Note that since $[0,1]-\cN$ is not closed, $\rho$ in general {\it can not} be extended to a smooth function on $[0,1].$

\begin{remark}
\label{RmNvsB}
The set $\cN$ is typically much smaller than the set $B$ used in \cite{Szewc}.
Indeed, if $f_t$ is a family of PEUMs satisfying a certain transversality condition then $B(f_t)$
contains an interval for almost all $t$ (see e.g. \cite{Sch1, Sch2}).
\end{remark}

The paper is organized as follows:

In Section \ref{Preliminary}, we give the necessary definitions. In particular, we introduce a special family of transfer operators used in the proof of Theorem \ref{Derivativesofrho}. We then prove several auxiliary facts of independent interest.

Section \ref{Repeated} starts with some explicit formulas for the first and second derivatives
\footnote{The derivatives are understood in the sense of Theorem \ref{Derivativesofrho}.}
 of $\rho$ which are proven to belong to $BV[0,1].$ Then we extend our analysis
 to repeated differentiation of arbitrary order proving Theorem \ref{Derivativesofrho}.

 Section \ref{Differentiability} begins with some results on the regularity of the saltus part of $\rho.$ \footnote{The density $\rho$ can be written as the sum of two functions, namely, the saltus part which is a sum of pure jumps
   and the regular part which is continuous.  Details are given in Section \ref{Differentiability}.}

Then we show that the regular part of $\rho$ is not only continuous but also absolutely continuous.
In the remaining subsections we prove Theorem \ref{ThW}. That is we show that $\rho$ admits
a Taylor expansion after we remove an exceptional set of zero Hausdorff dimension.

\section{Preliminaries}\label{Preliminary}

\subsection{Piecewise Expanding Unimodal Maps}\label{PEUM}

We work with mixing piecewise expanding unimodal maps.
$f:[0,1]\to [0,1]$ is a piecewise expanding unimodal map (PEUM) if there is a point $c$ called the critical point,
a number $\eps>0$ and a constant $\lambda>1$
such that
\begin{equation}\label{ExpUnimod}
f(x) = \begin{cases} f_1(x) & \mbox{if } x\leq c\\
f_2(x) & \mbox{if } x\geq c\end{cases}\end{equation}
where
$f_1$ is a $C^2$ map defined on $[0, c+\eps]$ and $f_2$ is a $C^2$ map
defined on $[c-\eps, 1]$ such that
$f_1(c)=f_2(c)$
and $|Df_j(x)|\geq \lambda$ for all $x$ from the domain of $f_j.$

PEUMs have unique a.c.i.m. \cite{LY} which is ergodic (see e.g. \cite{V}). Let us denote by $\rho$ the density of
the a.c.i.m.
$\rho$ is a function of bounded variation.

From now on, $\lambda$ will mean $\lambda := \displaystyle  \inf_{x\neq c} |Df (x)|.$

\subsection{Auxiliary facts and Transfer Operators}\label{Auxiliary}

Denote by $\xi(z)=\frac{D^2f(z)}{Df(z)}.$
In the arguments of this section we will need to  represent $D(|Df^my|)$ as a sum. Namely we have
$$ D(|Df^my|)=\frac{|Df^my|}{Df^my}\sum_{j=0}^{m-1}\xi(f^jy)Df^jy \;\;\mbox{ and }\;\;
D(Df^my)=\sum_{j=0}^{m-1}\xi(f^jy)Df^jy.$$
Both formulas are easy consequences of the chain rule.

We need to introduce a family of transfer operators acting on the space $BV[0,1]$ of functions of bounded variation.  $BV[0,1]$ it is a Banach space with the norm $\| \cdot \|_{BV}=\|\cdot\|_{\infty} + var(\cdot)$, where $\|\cdot\|_{\infty}$ is the usual supremum norm  and $var(\cdot)$ is the total variation
(cf. \cite{FitzRoyden}, page 116).


The first operator in our family is the Perron-Frobenius operator
$\mathcal{L}(\phi)(x) = \displaystyle \sum_{f(y)=x} \frac{\phi(y)}{|Df(y)|}.$

More generally, we shall use the following transfer operators acting on $BV[0,1]$.

\begin{definition}
   For $\phi \in BV[0,1]$, if $m\geq 1$, define the operator $\mathcal{L}_m(\phi)$ by

    $$\mathcal{L}_m(\phi)(x)=\sum_{f(y)=x}\frac{\phi(y)}{(Df(y))^m|Df(y)|}, $$ where $m$ is a nonnegative integer.
\end{definition}

\begin{definition}If $i,m \in \mathbb{N}$ and $h$ is a real-valued function, define $\mathfrak{D}^{i}_{m}(h) = \mathcal{L}_{m}^{i}(h)$.
    Let $k,\;i_1,\dots,i_k$ and $m_1>\dots>m_k$ be positive integers. For functions $h_1,\dots,h_k$,
    define $\mathfrak{D}_{m_1,\dots,m_k}^{i_1,\dots, i_k}$ at $(h_1,\dots,h_k)$ inductively by

    $$\mathfrak{D}_{m_1,\dots,m_k}^{i_1,\dots, i_k}(h_1,\dots,h_k)=\mathfrak{D}^{i_1}_{m_1}(h_1 \cdot \mathfrak{D}_{m_2,\dots,m_k}^{i_2,\dots, i_k}(h_2,\dots,h_k)).$$

\end{definition}

\begin{lemma} (see \cite[Lemma 3.8]{V})
\label{variationofDfim}
  There exists $C_1>0$ and $\bar{\lambda}_1>1$ such that for all $i,m \in \mathbb{N}$

  $$var\bigg(\frac{1}{|Df^i|^m} \bigg) \leq C_1 \bar{\lambda}_1^{-im},$$  where if $(Df^i)^m$ is not defined at $x$, then it is taken to be equal to $\frac{1}{2}(\lim_{y\to x+}(Df^i)^m(y)+\lim_{y \to x-}(Df^i)^m(y)).$
\end{lemma}

 We will use Lasota-Yorke inequlity (see e.g. \cite[Proposition 3.9]{V})
  saying that there exist constant $C_0>0$ and $\gamma<1$ such that
$$ var(\cL^n h)\leq C_0 \left[\|h\|_\infty+\gamma^n var(h)\right]. $$
Since $f$ is mixing, there is a constant $\theta<1$ such that

$$\mathcal{L}^n(h)=\bigg[ \int h(z)dz  \bigg]\rho(x)+O(\theta^n\|h\|_{BV}).$$
(see e.g.[1], Proposition 3.5, item 4).  In particular, we have that $\|\mathcal{L}^n(1)\|_{\infty}$ is bounded above by a constant $M$, which does not depends on $n$.
Then, we have the following:

\begin{proposition}
 \label{PrDecDh}
         \begin{itemize}\label{boundforD} \quad
           \item[(a)]
           $\|\mathfrak{D}^{i_1,i_2,\dots,i_k}_{m_1,m_2,\dots,m_k}(h_1,\dots,h_k)  \|_{\infty} \leq M^k (\lambda^{-i_1})^{m_1}(\lambda^{-i_2})^{m_2}\cdots(\lambda^{-i_k})^{m_k}\|h_1\|_{\infty}\|h_2\|_{\infty}\cdots \|h_k\|_{\infty}.$\\

          \item[(b)]
          There are constants $\bar{M}>0$ and $\bar{\lambda}>1$
           such that if $h_1,\dots h_k\in BV$ then
           $$ \|\mathfrak{D}^{i_1,i_2,\dots,i_k}_{m_1,m_2,\dots,m_k}(h_1,\dots,h_k)  \|_{BV}
           \leq \bar{M}  (\bar{\lambda}^{-i_1})^{m_1}
           (\bar{\lambda}^{-i_2})^{m_2}\cdots(\bar{\lambda}^{-i_k})^{m_k}
           \|h_1\|_{BV}\|h_2\|_{BV}\cdots \|h_k\|_{BV}.$$

           \end{itemize}
\end{proposition}

\begin{proof}
                 (a) We use induction on $k$. For $k=1$ we have

 $$       | \cL^i_m(h)(x) | = \bigg|\sum_{f^i y=x}\frac{h}{(Df^i (y))^m|Df^i(y)|} \bigg|
        \leq \sum_{f^i y=x}\frac{\|h\|_{\infty}}{|Df^i(y)|^m|DF^i (y)| }$$
$$        \leq \frac{\|h\|_{\infty}}{\lambda^{im}} \|\mathcal{L}^i(1)\|_{\infty }
        \leq \frac{M\|h\|_{\infty}}{\lambda^{im}}. $$

Now, let us suppose the result is true for $k-1$. Then, we have

        \begin{eqnarray*}
        |\mathfrak{D}^{i,i_2,\dots,i_k}_{m,m_2,\dots,m_k}(h,h_2,\dots,h_k)(x)| &=& | \mathcal{L}^{i}_m(h\;\mathfrak{D}^{i_2,\dots,i_k}_{m_2,\dots,m_k}(h_2,\dots,h_k))(x)|  \\
\leq        \bigg|\sum_{f^{i}y=x}\frac{h(y)\;\mathfrak{D}^{i_2,\dots,i_k}_{m_2,\dots,m_k}(h_2,\dots,h_k)(y)}{(Df^i(y))^m
|Df^i(y)|} \bigg|
        &\leq& \lambda^{-im}
        \sum_{f^i y=x}\frac{\|h\;\mathfrak{D}^{i_2,\dots,i_k}_{m_2,\dots,m_k}(h_2,\dots,h_k)\|_{\infty} }{|Df^i(y)|} \\
        \leq \lambda^{-im}\| h\;\mathfrak{D}^{i_2,\dots,i_k}_{m_2,\dots,m_k}(h_2,\dots,h_k)\|_{\infty} \|\mathcal{L}^i(1)\|_{\infty}
        &\leq &  M(\lambda^{-i})^m\|h\|_{\infty} \| \mathfrak{D}^{i_2,\dots,i_k}_{m_2,\dots,m_k}(h_2,\dots,h_k)\|_{\infty}. \\
       \end{eqnarray*}  Since we are assuming
        $$\| \mathfrak{D}^{i_2,\dots,i_k}_{m_2,\dots,m_k}(h_2,\dots,h_k)\|_{\infty} \leq M^{k-1} (\lambda^{-i_2})^{m_2}\cdots (\lambda^{-i_k})^{m_k}\|h_2\|_{\infty}\cdots \|h_k\|_{\infty},$$ the claim holds for $\mathfrak{D}^{i,i_2,\dots,i_k}_{m,m_2,\dots,m_k}(h,h_2,\dots,h_k)$.

     (b) We use induction again on $k$.   Let us prove the statement hold for $k=1$.  Note that

      $$var(\cL_m^i(h))=var\left(\cL^i\left(\frac{h}{(Df^i)^m}\right)\right),$$ where if $(Df^i)^m$ is not defined at $x$, then it is taken to be equal to $\frac{1}{2}(\lim_{y\to x+}(Df^i)^m(y)+\lim_{y \to x-}(Df^i)^m(y)).$

      A basic property of $var(\cdot)$ states that if $h,\tilde{h}\in BV[0,1]$, then
\begin{equation}
\label{VarMult}
     var(h\tilde{h})=var(|h|)\sup(\tilde{h})+var(\tilde{h})\sup(h).
\end{equation}

Combining \eqref{VarMult},  Lasota-Yorke inequality and Lemma \ref{variationofDfim} we obtain that there exist constants $C_0,C_1>0$ and $0<\gamma<1$ such that

      \begin{eqnarray}
        \nonumber
        var(\cL_m^i(h))&\leq&
            C_0\left[\gamma^i var\bigg(\frac{h}{(Df^i)^m}\bigg)+\bigg\| \frac{h}{(Df^i)^m} \bigg\|_{\infty}\right]\\
            &\leq& C_0\left[\gamma^i \;\| [(Df^i)]^{-m} \|_{\infty}var(h)+ \gamma^i var((|Df^i|)^{-m})\|h\|_{\infty}  +\bigg\| \frac{h}{(Df^i)^m} \bigg\|_{\infty}\right] \nonumber \\
\label{BoundVar}
            &\leq& C_0\bar{\lambda}^{-im}var(h)+C_0(C_1+1)\bar{\lambda}^{-im}\|h\|_{\infty}
            \leq C\bar{\lambda}^{-im} \|h\|_{BV},
      \end{eqnarray}
      where $C=\max\{C_0,C_0(C_1+1)\}$ and $\bar{\lambda}=\min\{ \lambda, \bar{\lambda}_1 \gamma_1   \}$.
      and the last inequality uses that
      $var(\cdot) \leq \|\cdot\|_{BV}$ and $\|\cdot\|_{\infty}\leq \|\cdot\|_{BV}$.
      \eqref{BoundVar} along with part (a) implies that
      $$\|\cL_m^i(h)\|_{BV}\leq \bar{M} \bar{\lambda}^{-im} \|h\|_{BV},$$ where $\bar{M}=\max\{C,M^{k}\}$.

      Now, assume the statement holds for $k-1$.  Then, under our assumption we have that

      $$\mathcal{D}_{m_2,\dots,m_k}^{i_2,\dots,i_k}(h_2,\dots,h_k)\leq \bar{M}_1\bar{\lambda}^{-i_2m_2}\cdots\bar{\lambda}^{-i_km_k}\|h_2\|_{BV}\cdots \|h_k \|_{BV},$$ for some constant $\bar{M}_1$.  Let us set $\mathcal{D}=\mathcal{D}_{m_2,\dots,m_k}^{i_2,\dots,i_k}(h_2,\dots,h_k)$.  Hence,

      \begin{eqnarray*}
        var(\mathcal{D}_{m,m_2,\dots,m_k}^{i,i_2,\dots,i_k}(h,h_2,\dots,h_k)) &=& var(\cL_m^i(h\mathcal{D})) \\
         &\leq& C\bar{\lambda}^{-im} (var(h\mathcal{D}) +\|h\mathcal{D} \|_{\infty} )\\
         &\leq& C\bar{\lambda}^{-im} (var(h)\|\mathcal{D}\|_{\infty}+\|h\|_{\infty}var(\mathcal{D})+\|h\|_{\infty}\|\mathcal{D}\|_{\infty})\\
         &\leq& C\bar{\lambda}^{-im}(\|h\|_{BV}\|\mathcal{D}\|_{BV}+\|h\|_{\infty}\|\mathcal{D}\|_{BV})\\
         &\leq& 2C\bar{\lambda}^{-im}(\|h\|_{BV}\|\mathcal{D}\|_{BV}).
      \end{eqnarray*}

      Using our inductive hypothesis, we finally obtain

       \begin{eqnarray*}\quad
      var(\mathcal{D}_{m,m_2,\dots,m_k}^{i,i_2,\dots,i_k}(h,h_2,\dots,h_k)) &\leq& \bar{M}\bar{\lambda}^{-im}\bar{\lambda}^{i_2m_2}\cdots \bar{\lambda}^{i_km_k} \|h\|_{BV}\|h_2\|_{BV}\cdots \|h_k\|_{BV},\\
      \end{eqnarray*}with $\bar{M}=2C\bar{M}_1$.  The above inequality along with part (a) proves part (b).
\end{proof}

If a series consisting of functions in $BV[0,1]$ converges to a function $g$, then the series of the derivatives of each term does not always converge to the derivative of $g$.
 However, assuming that the series of derivatives converges in $L^1$ we have the following result.

\begin{lemma}\label{limitforderivatives}
    If $\sum_{k=1}^{n}g_k\to g$ in $BV$ and $\sum_{k=1}^{n}g_k'\to h$ in $L_1$ then $g'=h$
    a.e.
\end{lemma}

\begin{proof}
    Let $\epsilon>0$.$\;$Then, there exists $N>0$ such that, for all $n \geq N$,

    $$\|g-\sum_{k=1}^ng_k\|_{BV} \leq \epsilon$$

    Since $\| f'  \|_{L_1} \leq \| f\|_{BV}$ for any function $f \in BV$, then

    $$\|g'- \sum_{k=1}^ng_k'\|_{L_1} \leq \|g-\sum_{k=1}^ng_k\|_{BV} \leq \epsilon$$

    Therefore, $\displaystyle \sum_{k=1}^ng_k'$ converges to $g'$ in $L_1$, hence $g'=h$ as claimed.
\end{proof}

Another simple but useful fact is the following.

\begin{lemma}\label{doubleseries}

  Let $g(s,r)$ be a function from $\mathbb{N}
  \times (\mathbb{N}\cup \{0\}) $ to $\mathbb{R}$. Suppose  that
  the series
  $\displaystyle \sum_{i=1}^{\infty} \sum_{j=0}^{i-1} |g(i-j,j)|$ converges. Then,
$\displaystyle \sum_{i=1}^{\infty}\sum_{j=0}^{i-1} g(i-j,j) = \sum_{c=1}^{\infty}\sum_{d=0}^{\infty} g(c,d)$
\end{lemma}
We leave the proof to the reader.\\

\section{Repeated Derivatives of the density function}\label{Repeated}

\subsection{Explicit formulas for the first and the second derivatives.}
Before  analyzing repeated derivatives of $\rho$ of arbitrary order, we will start by
giving explicit formulas for  $\rho'$ and $\rho''$.

Let us define

\begin{equation*}\label{formulaforrho1}
\rho_1=- \sum_{i=1}^{\infty} \mathcal{L}_1^i(\xi \cdot \rho)
\end{equation*}

Note that the series converges in $BV$ by Proposition \ref{PrDecDh} since $\rho$ and $\xi$ belong to $BV[0,1]$.

\begin{lemma}\label{firstderivative}
	(a) Let $\rho$ be the density of the invariant measure of $f$. Then,
	$\rho' =\rho_1$
	almost everywhere.

	(b) $(\cL^n 1)'(x)$ converges to $\rho_1(x)$ uniformly for $x$ which are not on the orbit of $c.$
\end{lemma}

\begin{proof}
Since $\rho$ is a fixed point of $\mathcal{L}$, then $\rho = \mathcal{L}^n(\rho)$ for all $n.$
Because $\rho$ is of bounded variation so is $\mathcal{L}^n(\rho)$,
hence both are differentiable almost everywhere.
In fact, differentiating both sides, we get $\rho'=(\cL^n \rho)'$ almost everywhere. Next if $h\in BV$ then

$$
(\cL^n h)'(x)= \sum_{f^ny=x}\frac{h'(y)}{Df^n(y)|Df^n(y)|}-\sum_{f^n(y)=x}\frac{h(y)\cdot D(|Df^n(y)|)}{|Df^n(y)|^2}
\;\;\;\;\mbox{ a. e.} $$

Note that
$$\left| \sum_{f^ny=x}\frac{h '(y)}{Df^n(y)|Df^n(y)|}\right|=\left|\cL^n_1(h')\right|\leq \lambda^{-n} \left(\cL^n(|h'|)\right)(x), $$
converges to $0$ in $L^1$ and almost everywhere.
Thus we focus on $\displaystyle \sum_{f^n(y)=x}\frac{h(y)\cdot D(|Df^n(y)|)}{|Df^n(y)|^2}$.
Assuming that $y \notin \{c,f(c),\dots, f^{n-1}(c)\}$ for each $y$ with $f^n y=x$ we have

$$
 \sum_{f^n(y)=x}\frac{h(y)\cdot D(|Df(y)|)}{|Df^n(y)|^2}=
 \sum_{f^n(y)=x}\frac{h(y)}{|Df^n(x)|^2}D\bigg( \prod_{a=0}^{n-1}|Df(f^ay)|\bigg)$$
$$= \sum_{f^n(y)=x} \frac{h(y)}{|Df^n(y)|^2} \frac{|Df^n(y)|}{Df^n(y)}\sum_{a=0}^{n-1}\xi(f^a(y)) Df^a(y)
= \sum_{f^n(y)=x} \frac{h(y)}{|Df^n(y)|}\sum_{a=0}^{n-1}\frac{\xi(f^a(y))}{Df^{n-a}(f^a(y))}$$
$$= \sum_{a=0}^{n-1} \sum_{f^{n-a}(z)=x} \frac{\xi(z)}{Df^{n-a}(z)}\sum_{f^a(y)=z}\frac{h(y)}{|Df^{n}(y)|}
= \sum_{a=0}^{n-1} \sum_{f^{n-a}(z)=x} \frac{\xi(z)}{Df^{n-a}(z)|Df^{n-a}(z)|}\sum_{f^a(y)=z}\frac{h(y)}{|Df^{a}(y)|}$$
$$= \sum_{a=0}^{n-1} \sum_{f^{n-a}(z)=x} \frac{\xi(z)}{Df^{n-a}(z)|Df^{n-a}(z)|}\mathcal{L}^a(h)(z)
= \sum_{a=0}^{n-1} \mathcal{L}^{n-a}_1(\xi\cL^a h)(x)$$
$$= \sum_{i=1}^{n} \mathcal{L}^i_1(\xi\cL^{n-i} h)(x)=\sum_{i=1}^\infty a_i(n) $$
where $a_i(n)=(\mathcal{L}^i_1(\xi\cL^{n-i} h))\chi_{i\leq n}.$
Proposition \ref{PrDecDh}.(a) shows that $|a_i(n)|\leq\dfrac{M^2}{\lambda^i}\times \|\xi \cL^{n-i}h\|_{\infty}$.  Since the second factor is less or equal than
  $M^2\|\xi \|_\infty \|h\|_{\infty}$, it follows that $|a_i(n)|\leq K \lambda^{-i}$ where $K$ does not depend
  on $n$ or $i.$
    Hence, applying Lebesgue's dominated convergence theorem (to integration with respect to the discrete measure)
  we can take the limit $n\to\infty$ term-by-term.
Since $$\displaystyle \lim_{n\to \infty} (\cL^{n-i} h)(x)=\left(\int_0^1 h(z) dz\right) \rho(x)$$
both parts (a) and (b) follow.
\end{proof}

At this point, we could get $\rho_2$  by differentiating each term in $(\ref{formulaforrho1})$. This is possible  due to Lemma \ref{limitforderivatives}.

\begin{proposition}\label{secondderivativeforrho}
The function $\rho_1$ is almost everywhere differentiable and
\begin{equation}\label{secondderivative}
\rho_1'= 3 \sum_{i=1}^{\infty} \sum_{j=1}^{\infty}\cL_2^i(\xi \cL_1^j(\xi \rho))+2  \sum_{i=1}^{\infty}\cL_2^{i}(\xi^2\rho)-\sum_{i=1}^{\infty}\cL_2^i(\xi'\rho)
\end{equation}

In particular, there exists $\rho_2\in BV$ such that
$\rho_1'=\rho_2$ almost everywhere.
\end{proposition}

\begin{proof}
By Lemma \ref{firstderivative}
$\rho_1=-\sum_{i=1}^{\infty} \cL_1^i(\xi\rho)$ almost everywhere.
Therefore by Lemma \ref{limitforderivatives}
$$\rho_1'(x)=- \sum_{i=1}^{\infty}\bigg( \sum_{f^iy=x}\frac{\xi(y)\rho(y)}{Df^i(y)|Df^i(y)|}\bigg)' =- \sum_{i=1}^{\infty} \sum_{f^iy=x}\bigg( \frac{\xi(y)\rho(y)}{Df^i(y)|Df^i(y)|}\bigg)' $$ almost everywhere. Decompose
$$\bigg( \frac{\xi(y)\rho(y)}{Df^i(y)|Df^i(y)|}\bigg)'= \underbrace{\frac{(\xi(y)\rho(y))'}{Df^i(y)|Df^i(y)|} }_{(I)}-\underbrace{\frac{\xi(y)\rho(y)(Df^i(y)|Df^i(y)|)'}{(Df^i(y))^2|Df^i(y)|^2}}_{(II)}.$$

Let us first work on $(I)$. We have

\begin{eqnarray*}
\sum_{i=1}^{\infty} \sum_{f^iy=x}(I) &=& \sum_{i=1}^{\infty} \sum_{f^iy=x}\frac{\xi'(y)Dy\rho(y)+\xi(y)\rho'(y)Dy}{Df^i(y)|Df^i(y)|}\\
&=& \sum_{i=1}^{\infty} \sum_{f^iy=x}\bigg( \frac{\xi'(y)\rho(y)}{Df^i(y)|Df^i(y)|} +\frac{\xi(y)\rho'(y)}{Df^i(y)|Df^i(y)|} \bigg)\\
&=& \sum_{i=1}^{\infty}\cL_2^i(\xi'\rho+\xi\rho')(x)\\
\end{eqnarray*}
By Lemma \ref{firstderivative}
$$ \sum_{i=1}^{\infty} \cL_2^i(\xi\rho')=-\sum_{i=1}^{\infty}
\sum_{j=1}^{\infty}\cL_2^i(\xi \cL_1^j(\xi \rho)).$$
Therefore
\begin{equation}\label{SecondDerivativeI}
\sum_{i=1}^{\infty} \sum_{f^iy=x}(I)=\sum_{i=1}^{\infty}\cL_2^i(\xi'\rho) -\sum_{i=1}^{\infty}
\sum_{j=1}^{\infty}\cL_2^i(\xi \cL_1^j(\xi \rho)).
\end{equation}

Now, let us analyze $(II)$.

$$\sum_{i=1}^{\infty} \sum_{f^iy=x}(II)=
\sum_{i=1}^{\infty} \sum_{f^iy=x}\frac{\xi(y)\rho(y)[D(Df^iy)|Df^iy|+D(|Df^iy|)(Df^i)(y)]}
{(Df^iy)^2|Df^iy|^2 }{(Df^iy)^2|Df^iy|^2} $$
$$=\sum_{i=1}^{\infty} \sum_{f^iy=x}\frac{\xi(y)\rho(y)\bigg[2|Df^iy|\sum_{j=0}^{i-1}\xi(f^jy)Df^jy\bigg] }{(Df^iy)^2|Df^iy|^2}
=2 \sum_{i=1}^{\infty}\sum_{j=0}^{i-1} \sum_{f^iy=x}\frac{\xi(y)\rho(y) Df^jy\xi(f^jy) }{(Df^iy)^2|Df^iy|} $$

Making the change of variable $z=f^jy$, we obtain

$$
\sum_{i=1}^{\infty} \sum_{f^iy=x}(II)
=2 \sum_{i=1}^{\infty}\sum_{j=0}^{i-1} \sum_{f^iy=x}\frac{\xi(y)\rho(y) (Df^j)(y)\xi(z) }{(Df^iy)^2|Df^iy|} $$
$$=2 \sum_{i=1}^{\infty}\sum_{j=0}^{i-1}
\sum_{f^iy=x}\frac{\xi(y)\rho(y) (Df^j)(y)\xi(z) }{(Df^{i-j}z)^2(Df^jy)^2|Df^{i-j}z||Df^jy|}
=2 \sum_{i=1}^{\infty}\sum_{j=0}^{i-1} \sum_{f^iy=x}\frac{\xi(y)\rho(y) \xi(z) }{(Df^{i-j}z)^2Df^jy|Df^{i-j}z||Df^jy|} $$
$$=2 \sum_{i=1}^{\infty}\sum_{j=0}^{i-1} \sum_{f^{i-j}z=x}\frac{ \xi(z) }{(Df^{i-j}z)^2|Df^{i-j}z|} \sum_{f^{i}y=z}\frac{\xi(y)\rho(y)}{Df^jy|Df^jy|}
=2 \sum_{i=1}^{\infty}\sum_{j=0}^{i-1}\cL_2^{i-j}(\xi\cL_1^j(\xi\rho)).$$

By Lemma \ref{doubleseries}
$$\sum_{i=1}^{\infty}\sum_{j=0}^{i-1}\cL_2^{i-j}(\xi\cL_1^j(\xi\rho))=\sum_{i=1}^{\infty}\sum_{j=0}^{\infty}\cL_2^{i-j}(\xi\cL_1^j(\xi\rho)).$$

Therefore
\begin{equation}\label{SecondDerivativeII}
\sum_{i=1}^{\infty} \sum_{f^iy=x}(II)=\sum_{i=1}^{\infty}\sum_{j=0}^{\infty}\cL_2^{i-j}(\xi\cL_1^j(\xi\rho)).
\end{equation}

Combining $(\ref{SecondDerivativeI})$ and $(\ref{SecondDerivativeII})$, we finally obtain

$$\rho_1' =-\sum_{i=1}^{\infty}\cL_2^i(\xi'\rho) + \sum_{i=1}^{\infty}\sum_{j=1}^{\infty}\cL_2^i(\xi \cL_1^j(\xi \rho))+  \sum_{i=1}^{\infty}\sum_{j=0}^{\infty}\cL_2^{j}(\xi\cL_1^j(\xi\rho))$$

$$= 3 \sum_{i=1}^{\infty} \sum_{j=1}^{\infty}\cL_2^i(\xi \cL_1^j(\xi \rho))+2  \sum_{i=1}^{\infty}\cL_2^{i}(\xi^2\rho))-\sum_{i=1}^{\infty}\cL_2^i(\xi'\rho)$$

almost everywhere as claimed.
\end{proof}

\subsection{Higher order derivatives.}

Lemma \ref{firstderivative} shows that $\rho'$ is in $BV$.
Then we saw in Proposition \ref{secondderivativeforrho} that $\rho_1'=\rho_2\in BV.$ Here we show that these results can be extended to repeated differentiation of arbitrary order.  We start with the following general result.

\begin{proposition}\label{derivativeforD}
    Let $k,\;i_1,\dots,i_k$ and $m_1>\dots>m_k$ be positive integers with $i_1,\dots,i_k \geq1$.$\;$
    Let $h_1,\dots, h_k$ be BV functions whose derivatives are in $L^\infty.$

    (a) The sum $\displaystyle \sum_{k\leq i_1+\cdots+i_k\leq n} \mathfrak{D}_{m,m_2,\dots,m_k}^{i_1,\dots, i_k}(h_1,\dots,h_k)$
      belongs to $BV[0,1]$ and if $n \geq 1$, its derivative is a finite sum of functions of the type\footnote{That is, the sums coincide
at the points where both of them are defined.}
    $\displaystyle \sum_{\widetilde{k}\leq \widetilde{i}_1+ \cdots +\widetilde{i}_{\widetilde{k}}\leq n} \mathfrak{D}_{m+1,\widetilde{m}_2,\dots,\widetilde{m}_{\widetilde{k}}}^{\widetilde{i}_1\dots, \widetilde{i}_{\widetilde{k}}}(\widetilde{h}_1,\dots,\widetilde{h}_{\widetilde{k}})$,
    where $k\leq \widetilde{k}\leq k+1$ ,$\widetilde{i}_1,\dots,\widetilde{i}_{\widetilde{k}}\geq1,$ $\widetilde{m}_1>\dots> \widetilde{m}_{\widetilde{k}} $ are positive integers  and $\widetilde{h}_1,\dots,\widetilde{h}_{\widetilde{k}} \in \{h_1,\dots,h_k,h_1',\dots,h_k',\xi,\xi'\}.$

    (b) The multiseries
    $$\displaystyle \sum_{i_1=1}^{\infty}\cdots \sum_{i_k=1}^{\infty} \mathfrak{D}_{m,m_2,\dots,m_k}^{i_1,\dots, i_k}(h_1,\dots,h_k)$$
    converges in $BV$ and its derivative equals almost everywhere to a finite sum of functions of the type
    $\displaystyle \sum_{\widetilde{i}_1} \cdots \sum_{\widetilde{i}_{\widetilde{k}}}
    \mathfrak{D}_{m+1,\widetilde{m}_2,\dots,\widetilde{m}_{\widetilde{k}}}^{
      \widetilde{i}_1\dots, \widetilde{i}_{\widetilde{k}}}(\widetilde{h}_1,\dots,\widetilde{h}_{\widetilde{k}})$,
    where $k\leq \widetilde{k}\leq k+1,$ $\widetilde{i}_1,\dots,\widetilde{i}_{\widetilde{k}}\geq 1,$ $\widetilde{m}_1>\dots> \widetilde{m}_{\widetilde{k}} $ are positive integers and $$\widetilde{h}_1,\dots,\widetilde{h}_{\widetilde{k}} \in \{h_1,\dots,h_k,h_1',\dots,h_k',\xi,\xi'\}.$$

\end{proposition}

\begin{proof}
  The sum $\displaystyle \sum_{k\leq i_1+\cdots+i_k\leq n} \mathfrak{D}_{m,m_2,\dots,m_k}^{i_1,\dots, i_k}(h_1,\dots,h_k)$
  is of bounded variation by Proposition \ref{PrDecDh}.
  To prove the rest of part (a), we use induction on $m$. For $m=1$, we need to compute

      $\displaystyle \sum_{i=1}^n D\bigg(\cL_1^i(h)\bigg) $, so let us work on $D\bigg(\cL_1^i(h)\bigg)$.$\;$Then
\begin{equation*}D\bigg(\cL_1^i(h)\bigg)(x)=\sum_{f^i(y)=x}D\bigg(\frac{h(y)}{Df^iy|Df^iy|}  \bigg)=\sum_{f^i(y)=x}\left[\frac{h'(y)}{Df^iy|Df^iy|}- \frac{h(y)D(Df^iy|Df^iy|)}{(Df^iy)^2|Df^iy|^2}\right]
\end{equation*}
$$=
    \sum_{f^i(y)=x}\left[\frac{h'(y)Dy}{Df^iy|Df^iy|}-\frac{h(y)(D(Df^iy)|Df^iy|+Df^iyD(|Df^iy|))}{(Df^iy)^2|Df^iy|^2}\right]$$
    $$=\sum_{f^i(y)=x} \left[
    \frac{h'(y)}{(Df^iy)^2|Df^iy|}-\frac{h(y)(2 |Df^iy|\sum_{j=0}^{i-1}\xi(f^j(x))Df^j(x) )}{(Df^iy)^2|Df^iy|^2}\right]$$
    $$=\cL_2^i(h')(x)-2    \sum_{j=0}^{i-1}\sum_{f^i(y)=x} \frac{h(y) \xi(f^j(y))Df^j(y))}{(Df^iy)^2|Df^iy|}.$$

    Let $z=f^j(y)$. Then
    \begin{eqnarray*}
    D\bigg(\cL_1^i(h)\bigg)(x)&=&\cL_2^i(h')(x)-2    \sum_{j=0}^{i-1}\sum_{f^i(y)=x} \frac{h(y) \xi(f^j(y))Df^j(y))}{(Df^iy)^2|Df^iy|}\\
    &=&\cL_2^i(h')(x)-2    \sum_{j=0}^{i-1}\sum_{f^i(y)=x} \frac{h(y) \xi(z)}{(Df^{i-j}z)^2|Df^{i-j}z| (Df^jy)|Df^jy|}\\
    &=&\cL_2^i(h')(x)-2    \sum_{j=0}^{i-1}\sum_{f^{i-j}z=x}  \frac{\xi(z)}{(Df^{i-j}z)^2|Df^{i-j}z|}  \sum_{f^jy=z} \frac{h(y)}{ (Df^jy)|Df^jy|}\\
    &=&\cL_2^i(h')(x)-2\sum_{j=0}^{i-1}\cL_2^{i-j}(\xi \cL_1^j(h)).
    \end{eqnarray*}
Hence
    \begin{eqnarray*}
    \sum_{i=1}^{n}D\bigg(\cL_1^i(h)\bigg)&=& \sum_{i=1}^{n}\cL_2^i(h')(x)-2\sum_{i=1}^{n}\sum_{j=0}^{i-1}\cL_2^{i-j}(\xi \cL_1^j(h))\\
    &=& \sum_{i=1}^{n}\cL_2^i(h')(x)-2\sum_{i=1}^{n}\cL_2^{i}(\xi h)-2\sum_{i=1}^{n}\sum_{j=1}^{i-1}\cL_2^{i-j}(\xi \cL_1^j(h))\\
    &=&\sum_{1}^{n}\cL_2^i(h')(x)-2\sum_{i=1}^{n}\cL_2^{i}(\xi h)-2\mathop{\sum_{2 \leq i+j\leq n}}_{1\leq i,1\leq j}\cL_2^{i}(\xi \cL_1^j(h))\\
    \end{eqnarray*}

    Therefore, the derivative is a finite sum of terms as described in the statement.

    Assume the statement is true for $l<m$.$\;$Let us prove that it also holds for $m$.
We are interested in the derivative of

    \begin{equation}\label{inductionform}
    \displaystyle \sum_{k+1\leq i+ i_1 + \cdots + i_k \leq n} \mathfrak{D}_{m,m_1,\dots,m_k}^{i,i_2,\dots, i_k}(h,h_1,\dots,h_k)
    \end{equation}
with $i\geq 1, i_1 \geq 1, \dots,i_k \geq 1$.
    For this, note that

    $$\sum_{k+1\leq i+i_1 \cdots i_k \leq n} \mathfrak{D}_{m,m_1,\dots,m_k}^{i,i_1,\dots, i_k}(h,h_1,\dots,h_k)=\sum_{i=1}^{n-k} \cL_m^{i}(h\;\sum_{k\leq i_1 + \cdots + i_k \leq n-i}\mathfrak{D}^{i_1,\dots,i_k}_{m_1,\dots,m_k}(h_1,\dots,h_k) ) $$

    Thus, if we are interested in the derivative of $(\ref{inductionform})$, we need to analyze

    $$\sum_{i=1}^{n-k} D\bigg[ \cL_m^{i}(h\;\sum_{k\leq i_1 + \cdots + i_k \leq n-i}\mathfrak{D}^{i_1,\dots,i_k}_{m_1,\dots,m_k}(h_1,\dots,h_k)  )\bigg] $$

    $$ =\sum_{i=1}^{n-k} \sum_{f^iy=x}D\bigg[\frac{h(y)\;\sum_{k\leq i_1 + \cdots + i_k \leq n-i}\mathfrak{D}^{i_1,\dots,i_k}_{m_1,\dots,m_k}(h_1,\dots,h_k)(y)  }{(Df^iy)^m|Df^iy|}\bigg]
    =\sum_{i=1}^{n-k} \sum_{f^iy=x} (I) - (II)$$ where

    $$(I)=\frac{D\bigg[h(y)\;\displaystyle \hsum
    \mathfrak{D}^{i_1,\dots,i_k}_{m_1,\dots,m_k}(h_1,\dots,h_k)(y)\bigg]}{(Df^iy)^m|Df^iy|},$$
         $$(II)= \frac{h(y)\;\displaystyle \hsum
    \mathfrak{D}^{i_1,\dots,i_k}_{m_1,\dots,m_k}(h_1,\dots,h_k)(y)D\bigg[(Df^iy)^{2m}|Df^iy|^2\bigg] }{(Df^iy)^{2m}|Df^iy|^2}  $$
and $\hsum$ means $\sum_{k\leq i_1 + \cdots + i_k \leq n-i}.$

    Let us first work on $(II)$. Note that

    $$
    D\bigg[ (Df^iy)^{m}|Df^iy| \bigg] = m (Df^iy)^{m-1}D(Df^iy)|Df^iy|+ (Df^iy)^{m}D(|Df^iy|)$$
    \begin{eqnarray*}
    &=& m (Df^iy)^{m-1}|Df^iy| \sum_{j=0}^{i-1}\xi(f^jy)Df^jy+ (Df^iy)^{m} \frac{|Df^iy|}{(Df^iy)}\sum_{j=0}^{i-1}\xi(f^jy)Df^jy \\
    &=& m (Df^iy)^{m-1}|Df^iy| \sum_{j=0}^{i-1}\xi(f^jy)Df^jy+ (Df^iy)^{m-1}|Df^iy|\sum_{j=0}^{i-1}\xi(f^jy)Df^jy \\
    &=& (m+1) (Df^iy)^{m-1}|Df^iy| \sum_{j=0}^{i-1}\xi(f^jy)Df^jy.
    \end{eqnarray*}
    Then $\displaystyle \sum_{i=1}^{n-k} \hsum
    \sum_{f^iy=x} (II)$ equals
$$ \sum_{i=1}^{n-k} \hsum
\sum_{f^iy=x} \frac{(m+1)h(y)\;\mathfrak{D}^{i_1,\dots,i_k}_{m_1,\dots,m_k}(h_1,\dots,h_k)(y) (Df^iy)^{m-1}|Df^iy| \sum_{j=0}^{i-1}\xi(f^jy)Df^jy}{(Df^iy)^{2m}|Df^iy|^2} $$
$$    = \sum_{i=1}^{n-k}\sum_{j=0}^{i-1}\hsum
\sum_{f^iy=x} \frac{(m+1)h(y)\;\mathfrak{D}^{i_1,\dots,i_k}_{m_1,\dots,m_k}(h_1,\dots,h_k)(y) \xi(f^jy)Df^jy}{(Df^iy)^{m+1}|Df^iy|} . $$

    Let $z=f^jy$. Then $\displaystyle \sum_{i=1}^{n-k} \hsum
    \sum_{f^iy=x} (II)$ equals

    \begin{eqnarray*}
    &=& \sum_{i=1}^{n-k} \sum_{j=0}^{i-1} \hsum
    \sum_{f^iy=x} \frac{(m+1)h(y)\;\mathfrak{D}^{i_1,\dots,i_k}_{m_1,\dots,m_k}(h_1,\dots,h_k)(y) \xi(f^jy)Df^jy}{(Df^iy)^{m+1}|Df^iy|} \\
    &=& \sum_{i=1}^{n-k}\sum_{j=0}^{i-1}\hsum
    \sum_{f^iy=x} \frac{(m+1)h(y)\;\mathfrak{D}^{i_1,\dots,i_k}_{m_1,\dots,m_k}(h_1,\dots,h_k)(y) \xi(z)Df^jy}{(Df^{i-j}z)^{m+1}|Df^{i-j}z|(Df^{j}y)^{m+1}|Df^{j}y|} \\
    &=& \sum_{i=1}^{n-k}\sum_{j=0}^{i-1}\hsum
    \sum_{f^iy=x} \frac{(m+1)h(y)\;\mathfrak{D}^{i_1,\dots,i_k}_{m_1,\dots,m_k}(h_1,\dots,h_k)(y) \xi(z)}{(Df^{i-j}z)^{m+1}|Df^{i-j}z|(Df^{j}y)^{m}|Df^{j}y|} \\
      \end{eqnarray*}
       \begin{eqnarray*}
     &=& \sum_{i=1}^{n-k}\sum_{j=0}^{i-1}\hsum
     \sum_{f^{i-j}z=x}\frac{\xi(z)}{(Df^{i-j}z)^{m+1}|Df^{i-j}z|}
    \sum_{f^{j}y=z}  \frac{(m+1)h(y)\;\mathfrak{D}^{i_1,\dots,i_k}_{m_1,\dots,m_k}(h_1,\dots,h_k)(y) }{(Df^{j}y)^m|Df^{j}y|} \\
        &=&(m+1)\sum_{i=1}^{n-k}\sum_{j=0}^{i-1}\hsum
        \cL_{m+1}^{i-j}( \xi\; \cL_m^{j}(h\;\mathfrak{D}^{i_1,\dots,i_k}_{m_1,\dots,m_k}(h_1,\dots,h_k)) )(x) \\
    &=& (m+1)\sum_{i=1}^{n-k}\hsum
    \cL_{m+1}^{i}( \xi\;(h\;\mathfrak{D}^{i_1,\dots,i_k}_{m_1,\dots,m_k}(h_1,\dots,h_k) ))(x)+\\
    & &(m+1) \mathop{\sum_{1\leq i+j \leq n-k}}_{1\leq i,1\leq j} \hsum
    \cL_{m+1}^{i-j}( \xi\; \cL_m^{j}(h\;\mathfrak{D}^{i_1,\dots,i_k}_{m_1,\dots,m_k}(h_1,\dots,h_k)))(x)
   =A+B .\\
    \end{eqnarray*}
The last two terms can be rewritten as
    $$A=(m+1)\sum_{1\leq i+ i_1 + \cdots + i_k \leq n}  \cL_{m+1}^{i}( \xi\; (h\;\mathfrak{D}^{i_1,\dots,i_k}_{m_1,\dots,m_k}(h_1,\dots,h_k) ))(x),$$

$$B=(m+1) \sum_{k+2 \leq i+j+i_1 + \cdots + i_k \leq n}  \cL_{m+1}^{i}( \xi\; \cL_m^{j}(h\;\mathfrak{D}^{i_1,\dots,i_k}_{m_1,\dots,m_k}(h_1,\dots,h_k)) )(x) . $$

    Therefore $\sum_{i=1}^{n-k} \sum_{f^iy=x} (II) $ is a sum of terms described in the statement.

    Now, let us analyze $(I)$. Note that $\sum_{i=1}^{n-k} \sum_{f^iy=x} (I)$ equals to

    $$\sum_{i=1}^{n-k} \sum_{f^iy=x} \frac{h'(y)Dy\hsum
    \mathfrak{D}^{i_1,\dots,i_k}_{m_1,\dots,m_k}(h_1,\dots,h_k)(y) }{(Df^iy)^m|Df^iy|}
    - \frac{h(y)D\bigg[\hsum
    \mathfrak{D}^{i_1,\dots,i_k}_{m_1,\dots,m_k}(h_1,\dots,h_k)(y)\bigg]}{(Df^iy)^m|Df^iy|}$$
    $$= \sum_{i=1}^{n-k} \sum_{f^iy=x} \frac{h'(y)\hsum
    \mathfrak{D}^{i_1,\dots,i_k}_{m_1,\dots,m_k}(h_1,\dots,h_k)(y) }{(Df^iy)^{m+1}|Df^iy|}
    - \frac{h(y)D\bigg[\hsum
    \mathfrak{D}^{i_1,\dots,i_k}_{m_1,\dots,m_k}(h_1,\dots,h_k)(y)\bigg]}{(Df^iy)^m|Df^iy|}$$
    $$=\sum_{i=1}^{n-k} \cL_{m+1}^i\bigg(h'\;
    \hsum
    \mathfrak{D}^{i_1,\dots,i_k}_{m_1,\dots,m_k}(h_1,\dots,h_k)\bigg)(y)
    - \sum_{i=1}^{n-k} \sum_{f^iy=x} \frac{h(y)\hsum
    D\bigg[\mathfrak{D}^{i_1,\dots,i_k}_{m_1,\dots,m_k}(h_1,\dots,h_k)(y)\bigg]}{(Df^iy)^m|Df^iy|} . $$

    Using our inductive hypothesis, the derivative of $\hsum
    \mathfrak{D}^{i_1,\dots,i_k}_{m_1,\dots,m_k}(h_1,\dots,h_k)$ is the finite sum of terms of the type $\sum_{\widetilde{k}\leq \widetilde{i}_1 + \cdots + \widetilde{i}_{\widetilde{k}} \leq n-i}\mathfrak{D}^{\widetilde{i}_1,\widetilde{i}_2,\dots,\widetilde{i}_{\widetilde{k}}}_{m_1+1,\widetilde{m}_2, \dots ,\widetilde{m}_{\widetilde{k}}}(\widetilde{h}_1,\dots,\widetilde{h}_{\widetilde{k}})$.$\;\;$Hence, let us take one of these terms and analyze the expression

$$     \sum_{i=1}^{n-k} \sum_{\widetilde{k}\leq \widetilde{i}_1 + \cdots + \widetilde{i}_{\widetilde{k}} \leq n-i} \sum_{f^iy=x} \frac{h(y)\mathfrak{D}^{i_1+1,\widetilde{i}_2,\dots,\widetilde{i}_{\widetilde{k}}}_{m_1+1,\widetilde{m}_2, \dots ,\widetilde{m}_{\widetilde{k}}}(\widetilde{h}_1,\dots,\widetilde{h}_{\widetilde{k}})(y)\cdot Dy}{(Df^iy)^m|Df^iy|} $$
$$= \sum_{i=1}^{n-k} \sum_{\widetilde{k}\leq \widetilde{i}_1 + \cdots + \widetilde{i}_{\widetilde{k}} \leq n-i} \sum_{f^iy=x}\frac{h(y)\mathfrak{D}^{i_1+1,\widetilde{i}_2,\dots,\widetilde{i}_{\widetilde{k}}}_{m_1+1,\widetilde{m}_2, \dots ,\widetilde{m}_{\widetilde{k}}}(\widetilde{h}_1,\dots,\widetilde{h}_{\widetilde{k}})(y)}{(Df^iy)^{m+1}|Df^iy|} $$
    $$=\sum_{i=1}^{n-k} \sum_{\widetilde{k}\leq \widetilde{i}_1 + \cdots + \widetilde{i}_{\widetilde{k}} \leq n-i} \cL_{m+1}^{i}(h\;\mathfrak{D}^{i_1+1,\widetilde{i}_2,\dots,\widetilde{i}_{\widetilde{k}}}_{m_1+1,\widetilde{m}_2, \dots ,\widetilde{m}_{\widetilde{k}}}(\widetilde{h}_1,\dots,\widetilde{h}_{\widetilde{k}}) )$$
    $$=\sum_{\widetilde{k}+1 \leq 1+\widetilde{i}_1 + \cdots + \widetilde{i}_{\widetilde{k}} \leq n}\cL_{m+1}^{i}(h\;\mathfrak{D}^{i_1+1,\widetilde{i}_2,\dots,\widetilde{i}_{\widetilde{k}}}_{m_1+1,\widetilde{m}_2, \dots ,\widetilde{m}_{\widetilde{k}}}(\widetilde{h}_1,\dots,\widetilde{h}_{\widetilde{k}})).$$

Since we have a finite sums of terms as
above, we obtained that our proposition also holds for $m$. Therefore, part (a)  is established by induction.

(b) 
  Proposition \ref{PrDecDh}
  allows us to take the limit $n\to\infty.$
 Then the condition $\tilde{k}\leq \tilde{i}_1 + \cdots + \tilde{i}_{\tilde{k}} \leq n$
becomes
$\tilde{k}\leq \tilde{i}_1 + \cdots + \tilde{i}_{\tilde{k}} \leq \infty$ and using the condition
$\tilde{i}_1 \geq 1, \dots, \tilde{i}_{\tilde{k}} \geq 1$
the sum
$$\displaystyle \sum_{\widetilde{i}_1 \geq 1, \dots, \widetilde{i}_{\tilde{k}}\geq 1,\;
  \widetilde{k}\leq \widetilde{i}_1+ \cdots +\widetilde{i}_{\widetilde{k}}\leq n}
\mathfrak{D}_{m+1,\widetilde{m}_2,\dots,\widetilde{m}_{\widetilde{k}}}^{\widetilde{i}_1\dots, \widetilde{i}_{\widetilde{k}}}(\widetilde{h}_1,\dots,\widetilde{h}_{\widetilde{k}})$$
converges to
$$\sum_{\widetilde{i}_1=1}^{\infty}\cdots \sum_{\widetilde{i}_{\tilde{k}}=1}^{\infty}
\mathfrak{D}_{m+1,\widetilde{m}_2,\dots,\widetilde{m}_{\widetilde{k}}}^{\widetilde{i}_1\dots, \widetilde{i}_{\widetilde{k}}}(\widetilde{h}_1,\dots,\widetilde{h}_{\widetilde{k}})$$
in $L^\infty.$ Likewise the sum
$$\sum_{i_1 \geq 1, \dots, i_{k}\geq 1,\;
  k\leq i_1+ \cdots +i_k\leq n}
\mathfrak{D}_{m+1,m_2,\dots,m_k}^{i_1\dots, i_{k}}(h_1,\dots,h_{k})$$
converges to
$$\sum_{i_1=1}^{\infty}\cdots \sum_{{i_k}=1}^{\infty}
\mathfrak{D}_{m+1,m_2,\dots,m_k}^{i_1\dots, i_k}(h_1,\dots, h_k)$$
in $BV.$
Therefore part (b) follows from part (a) and Lemma \ref{limitforderivatives}.
    \end{proof}

Proposition \ref{derivativeforD}(b) allows us to derive Theorem \ref{Derivativesofrho}.

\begin{proof}[Proof of Theorem \ref{Derivativesofrho}]


   Let $\cB_p$ be the space of functions which are $C^p$ away from $c.$

    We proceed by induction.  Cases $j=1$ and $j=2$ were already handled in Lemma \ref{firstderivative} and
Proposition \ref{secondderivativeforrho} respectively.
Assume the claim holds for $j-1$ and
moreover that $\rho_{j-1}$ is of the form
\begin{equation}
  \label{Ansatz-k-1}
\rho_{j-1}= \sum_{finite}\sum_{i_1,\dots,i_s \geq 1}
\mathfrak{D}^{i_1,\dots,i_s}_{j-1,m_2,\dots,m_s}(h_1,\dots, h_{s-1},h_s),
\end{equation}
where $s\geq 1,\; j>m_2>\cdots>m_s$,
$h_s=\hat{h} \rho$
and $h_1,\dots,h_{s-1}, \hat{h}$ are in $\cB_{k-j+2}.$
Let us prove the same for $j$.

By Proposition \ref{derivativeforD}(b), $\rho_j$ has the form

$$\rho_{j}= \sum_{finite}\sum_{\widetilde{i}_1,\dots,\widetilde{i}_r \geq 1}
\mathfrak{D}^{\widetilde{i}_1,\dots,\widetilde{i}_s}_{j,\widetilde{m}_2,\dots,\widetilde{m}_r}(\widetilde{h}_1,\dots,\widetilde{h}_{r-1}, \widetilde{h}_r),$$
where $s\leq r \leq s+1$  and for each $1\leq l \leq r$,
$\widetilde{h}_l \in B=\{h_1,\dots,h_s,h_1',\dots,h_s',\xi,\xi' \}.$
Next for the terms which contain
$h_s'=(\hat{h}')\rho+\hat{h}\rho_1$ we can use
Lemma \ref{firstderivative} to express $\rho_1$ in terms of $\rho$ the same way as we did
in the proof of Proposition \ref{secondderivativeforrho}. It follows that $\rho_j$ is of the form
\eqref{Ansatz-k-1}.
Theorem \ref{Derivativesofrho} is thus proven by induction.

\end{proof}


\section{Differentiability set for the density.}\label{Differentiability}
\subsection{Saltus part.}\label{Saltus}
Any function of bounded variation $\phi$ can be decomposed as
$$\phi=\phi_r + \phi_s$$
where $\phi_r$ is a continuous function, called the regular part, and $\phi_s$
is constant except at discontinuities of $\phi.$
$\phi_s$ is called the saltus part, it is discontinuous on a countable set (see \cite{RieszSzNagy}, page 14)

In fact, in the case of $\rho$, $\rho_s$ can be explicitly written as (\cite{Bal1})

$$\rho_s = \sum_{j\geq 1} \alpha_j H_{c_j}$$ where $c_j=f^j(c)$, $\alpha_j= \displaystyle \lim_{x\uparrow c_j}\rho(x)-\lim_{x\downarrow c_j}\rho(x)$ and $H_{c_j}$ is defined as

 \begin{equation}\label{Hfunction}
H_{c_j}(x)=
\begin{cases}
1 & \text{if  $x<c_j$ } \\
\frac{1}{2} & \text{if  $x=c_j$ } \\
0 & \text{if  $x>c_j$ } \\
\end{cases}
\end{equation}

\begin{lemma}
\label{LmSaltRho}
If $c$ is not periodic then
$$ \alpha_j=\pm \rho(c)\left[\frac{1}{|Df^j_+(c)|}+\frac{1}{|Df^j_-(c)|}\right], $$ where the expression takes the sign $+$ (resp. the sign $-$) if $f$ has a maximum (resp. minimum) at $c$.

\end{lemma}

\begin{proof}
We have
$$\alpha_j=\lim_{x\uparrow c_j}\rho(x)-\lim_{x\downarrow c_j}\rho(x).$$

Using the fact that $\rho$ is a fixed point of $\mathcal{L}$ and $\displaystyle \mathcal{L}^j\rho(x)=\sum_{f^jy=x}\frac{\rho(y)}{Df^j(y)}$, we can see that $\rho$ has a discontinuity at $x=c_j$. $\;\;$In fact, among all the $y's$ in the set $\{f^{-j}c_j\}$, the discontinuity comes from $y=c$, therefore

$$\alpha_{j}=\lim_{y\uparrow c}\frac{\rho(y)}{Df^j(y)}-\lim_{y\downarrow c}\frac{\rho(y)}{Df^j(y)}. \qedhere $$
\end{proof}

\begin{proposition}
    For $k\geq 0$, the element $\rho_k$ of the sequence from Theorem \ref{Derivativesofrho}
    can be decomposed as $(\rho_k)_r+(\rho_k)_s$, where $(\rho_k)_r$ is a continuous function and $(\rho_k)_s =\sum_{m\geq 1}\alpha_{k,j}H_{c_j}$, with $H_{c_j}$ defined in $(\ref{Hfunction})$ and $\displaystyle \alpha_{k,j}=\lim_{x\uparrow c_j}\rho_k(x)-\lim_{x\downarrow c_j}\rho_k(x).$ Moreover there exists $\theta<1$ such that
$|\alpha_{k,j}| \leq K\theta^{j}$
    \end{proposition}

\begin{proof}
        The existence of decomposition follows from the fact that, due to Theorem \ref{Derivativesofrho},
         $\rho_k\in BV$-function. We need to show that all discontinuities of $\rho_k$ lie on the critical orbit
         and bound the size of discontinuity.

         Let $z$ be a discontinuity point of $\rho_k$ which is different from $c_i$ for $i=1\dots j.$
         Let $\brrho=\cL^j (1).$
        In the proof of Proposition \ref{Derivativesofrho} we saw that

        $$\rho_k=
        \sum_{finite}\sum_{i,i_2,\dots,i_s \geq 1}
        \mathfrak{D}^{i,i_2\dots,i_s}_{k,m_2,\dots,m_k}(h_1,\dots,h_{s-1}, \rho)$$
        $$=\sum_{finite}\sum_{i,i_2,\dots,i_s \geq 1}
        \mathfrak{D}^{i,i_2\dots,i_s}_{k,m_2,\dots,m_k}(h_1,\dots,h_{s-1}, \brrho)+
        \sum_{finite}\sum_{i,i_2,\dots,i_s \geq 1}
        \mathfrak{D}^{i,i_2\dots,i_s}_{k,m_2,\dots,m_k}(h_1,\dots,h_{s-1}, \rho-\brrho).$$
Denote $\Delta(h)=\lim_{x\uparrow z}h(x) - \lim_{x \downarrow z} h(x).$ Then
$$\Delta\left(\sum_{finite}\sum_{i,i_2,\dots,i_s \geq 1}
        \mathfrak{D}^{i,i_2\dots,i_s}_{k,m_2,\dots,m_k}(h_1,\dots,h_{s-1}, \rho-\brrho)\right)=O(\theta^j)$$
in view of Proposition \ref{PrDecDh}
and the fact that $\rho-\brrho=O(\theta^j).$

        Note that if $i,i_2,\dots,i_s < j$ then $\bigg(\stackrel{k}{\cL_1}\bigg)^{i}$ and         $\bigg(\stackrel{m_{r}}{\cL_1}\bigg)^{i_r}$ are continuous at $z$ for $r=2,\dots,s$, so

        $$\sum_{finite}
        \Delta\left(\sum_{i,i_2,\dots,i_k<j}  \mathfrak{D}^{i,i_2\dots,i_s}_{k,m_2,\dots,m_s}(h_1,\dots,h_{s-1}, \brrho)\right)
        =0   . $$

Applying Proposition \ref{PrDecDh}  again we see that
$$ \sum_{finite}\sum_{\max(i,i_2,\dots,i_s)>j}  \mathfrak{D}^{i,i_2\dots,i_s}_{k,m_2,\dots,m_s}(h_1,\dots,h_{s-1}, \brrho)=
O\left(\sum_{\max(i,i_2,\dots,i_s)>j} \lambda^{-(i+i_2+\dots i_s)}\right)$$ and since the expression in the right side
is $O\left(j^{s} \lambda^{-j}\right)$, we have

$$ \Delta\left(
\sum_{finite}\sum_{\max(i,i_2,\dots,i_s)>j}  \mathfrak{D}^{i,i_2\dots,i_s}_{k,m_2,\dots,m_s}(h_1,\dots,h_{s-1}, \brrho)
\right)\leq O\left(j^{s} \lambda^{-j}\right). $$

In particular if $z$ is not on the critical orbit then $\Delta\rho_k=0$ and if $z=c_j$ then $\Delta\rho_k$ is exponentially
small in $j$ as claimed.
    \end{proof}

\subsection{Absolute continuity.}\label{Absolutely}
As we mentioned before, the regular part of $\rho$ is continuous.
In fact, it is absolutely continuous.

\begin{theorem}
\label{ThAC}
The regular part of $\rho$ is absolutely continuous. That is
$$ \rho_r(x_2)-\rho_r(x_1)=\int_{x_1}^{x_2} \rho'(x) dx. $$
\end{theorem}

\begin{proof}
    Let $n \geq 1$ and let $x_2,x_1 \in [0,1]$. Then
\begin{equation}
\label{ACTrOp}
(\mathcal{L}^n(1))(x_2)-(\mathcal{L}^n(1))(x_1)
\end{equation}
$$ = \int_{x_1}^{x_2}(\mathcal{L}^n(1))'(x)dx+\mathop{\sum_{j\leq n}}_{c_j \in [x_1,x_2]} \triangle_j(\mathcal{L}^n(1)),$$
where
$\triangle_j(\mathcal{L}^n(1))=\lim_{x\uparrow c_j}\mathcal{L}^n(1)(x) - \lim_{x\downarrow c_j}\mathcal{L}^n(1)(x)$.

As $n\to \infty$, $(\mathcal{L}^n(1))(x) \to \rho(x)$.$\;$Hence, $\triangle_j(\mathcal{L}^n(1)) \to \triangle_j \rho$.
By Lemma \ref{firstderivative}
    $(\mathcal{L}^n(1))' \to \rho_1$ as $n \to \infty.$ Thus
letting $n \to \infty$ in \eqref{ACTrOp} we get

    \begin{eqnarray*}
    \rho(x_2)-\rho(x_1) &=& \int_{x_1}^{x_2}\rho_1(x)dx + \sum_{c_j \in [x_1,x_2]}\triangle_j \rho\\
    &=& \int_{x_1}^{x_2} \rho_1(x)dx + \rho_s(x_2)-\rho_s(x_1)\\
    \end{eqnarray*}

Therefore
$\rho_r(x_2)-\rho_r(x_1) =\int_{x_1}^{x_2} \rho_1(x)dx. $
\end{proof}

\begin{proposition}\label{IneqRho}
There exist constants $K\geq1$, $D\geq1$ and $\varsigma<1$ such that if $\bar{x}$ satisfies
\begin{equation}\label{cjawayfromxbar}
d(c_j,\bar{x})>\epsilon,
\end{equation}for $j \leq n$ and
$d(x,\bar{x})<\epsilon$, then
$$|\rho(x)-\rho(\bar{x})| \leq K\epsilon+D\varsigma^n. $$
\end{proposition}

\begin{proof}
    Decompose
    \begin{equation}\label{DecompositionofRho}
    \rho(x)-\rho(\bar{x}) =  (\rho_r(x)-\rho_r(\bar{x})) +(\rho(x)_s-\rho_s(\bar{x})).
    \end{equation}

Combining Theorem \ref{ThAC} with the fact that $\rho'=\rho_1 \in BV[0,1]$, we get
    \begin{equation}\label{RegularDecomposition}
    |\rho_r(x)-\rho_r(\bar{x})| \leq K\epsilon.
    \end{equation}

Also, $(\ref{cjawayfromxbar})$ implies

    \begin{equation}\label{SaltusDecomposition}
    \rho_s(x)-\rho_s(\bar{x}) = \sum_{j \geq n} \alpha_j[H_{c_j}(x)-H_{c_j}(\bar{x})].
    \end{equation}

By Lemma \ref{LmSaltRho}
$\displaystyle |\alpha_j| \leq \frac{2\| \rho\|_{\infty}}{\lambda^j}.$
Hence, we can bound \eqref{SaltusDecomposition} as

    $$|\rho_s(x)-\rho_s(\bar{x})| \leq \sum_{j \geq n}|\alpha_j|\bigg| H_{c_j}(x)-H_{c_j}(\bar{x}) \bigg|
    \leq 2\|\rho \|_{\infty} \sum_{j \geq n} \frac{1}{\lambda^j}$$
    $$= 2\|\rho \|_{\infty} \frac{1}{\lambda^n} \sum_{j\geq 1} \frac{1}{\lambda^j}
    = 2\|\rho \|_{\infty}\bigg(\frac{\lambda}{\lambda-1} \bigg) \frac{1}{\lambda^n} $$

Taking $D=2\|\rho \|_{\infty}\bigg(\frac{\lambda}{\lambda-1} \bigg)$, $\varsigma=\frac{1}{\lambda}$, we have

    \begin{equation}\label{SaltusDecomposition2}
    |\rho_s(x)-\rho_s(\bar{x})| \leq D \varsigma^n.
    \end{equation}

Combining \eqref{DecompositionofRho}, \eqref{RegularDecomposition} and \eqref{SaltusDecomposition2}
we obtain the result.
\end{proof}

\subsection{Differentiability points.}\label{DifferentiabilityPoints}
Recall that since $f$ is mixing, then there exists a constant $\theta<1$ such that
$$ \cL^n h=\left[\int h(z) dz\right] \rho(x)+O\left(\theta^n ||h||_{BV}\right).$$

 \begin{theorem}\label{higherderivative}
If $1>\beta>\max(\theta, 1/\lambda)$ and if
$\bar{x}$ is a point such that
$d(\bar{x},c_j) \geq \beta^j$ for all $j \geq j_0$ then $\rho_k$ is differentiable at $\bar{x}$.
\end{theorem}

\begin{proof}
    Let $\epsilon>0$ and let $x$ such that $d(x,\bar{x})=\epsilon$.

    Let $n$ be the maximal number such that
\begin{equation}
\label{NonCritInt}
    c_j \notin [x;\bar{x}] \text{ for all }j\leq n.
\end{equation}
     Then
    $\epsilon \geq \beta^n,$ hence
    $\epsilon\lambda^n \geq \beta^n \lambda^n$ and
    $\frac{\epsilon}{\theta^n} \geq \frac{\beta^n}{\theta^n}. $

    By definition of $\beta$, $\beta\lambda >1$ and $\frac{\beta}{\theta}>1.$ Hence, $\beta^n \lambda^n \to \infty$ and $\frac{\beta^n}{\theta^n} \to \infty$ as $n \to \infty$.$\;$Therefore,

    $$\epsilon\lambda^{n}\to\infty$$ and $$\frac{\epsilon}{\theta^n} \to \infty.$$ as $n \to \infty$.

By Theorem \ref{Derivativesofrho}

    $$\rho_k(x)=\sum_{finite}\sum_{i_1,\dots,i_k=1}^{\infty}\mathfrak{D}^{i_1,\dots,i_k}_{m_1,\dots,m_k}(h_1,\dots,h_k,\rho).$$

Let $\brrho=\cL^n (1).$
    Since $\rho=\brrho+O(\theta^{n})$, Proposition \ref{PrDecDh} implies that
    we can write the above expression as

    $$\rho_k(x)=\sum_{finite}\mathop{\sum_{k\leq i_1,\dots,i_k<n}}_{1\leq i_1,\dots,1\leq i_k}\mathfrak{D}^{i_1,\dots,i_k}_{m_1,\dots,m_k}(h_1,\dots,h_k, \brrho)+
    O\left(\lambda^{-n}+\theta^n\right).$$

    Therefore

    \begin{equation}
    \rho_k(x)-\rho_k(\bar{x})=
     \end{equation}
$$    \mathop{\sum_{k\leq i_1,\dots,i_k<n}}_{1\leq i_1,\dots,1\leq i_k}\mathfrak{D}^{i_1,\dots,i_k}_{m_1,\dots,m_k}(h_1,\dots,h_k, \brrho)(x)-\mathfrak{D}^{i_1,\dots,i_k}_{m_1,\dots,m_k}(h_1,\dots,h_k,\brrho)(\bar{x})+
O\left(\lambda^{-n}+\theta^n\right).$$

    Note that $\mathfrak{D}^{i_1,\dots,i_k}_{m_1,\dots,m_k}(h_1,\dots,h_k, \brrho)$ is differentiable in $[x;\bar{x}]$
    since $h_1\dots h_k$ are $C^1$ away from $c$ and \eqref{NonCritInt} ensures that $f^{-n}[x, \bar{x}]$ does not contain $c.$

    Thus

    $$\mathfrak{D}^{i_1,\dots,i_k}_{m_1,\dots,m_k}(h_1,\dots,h_k, \brrho)(x)-
    \mathfrak{D}^{i_1,\dots,i_k}_{m_1,\dots,m_k}(h_1,\dots,h_k, \brrho)(\bar{x})$$
    \begin{equation} \label{ApproxRho'}
    =\int_{\bar{x}}^x \bigg( \mathfrak{D}^{i_1,\dots,i_k}_{m_1,\dots,m_k}(h_1,\dots,h_k, \brrho) \bigg)'(s)ds
    \end{equation}

    By Proposition \ref{derivativeforD}

    $$\bigg(\mathop{\sum_{k\leq i_1,\dots,i_k<n}}_{1\leq i_1,\dots,1\leq i_k} \mathfrak{D}^{i_1,\dots,i_k}_{m_1,\dots,m_k}(h_1,\dots,h_k, \brrho) \bigg)'=
    \sum_{finite}\sum_{\widetilde{i}_1,\dots,\widetilde{i}_k}\mathfrak{D}^{\widetilde{i}_1,\dots,\widetilde{i}_k}_{m_1+1,\dots,\widetilde{m}_{\widetilde{k}}}(\widetilde{h}_1,\dots,\widetilde{h}_{n,k}\Upsilon_n),$$ where $\widetilde{h}_1,\dots,\widetilde{h}_{k} \in \{h_1,h_2,\dots,h_{k},h_1',\dots,h_{k}',\xi,\xi'\}$ and $\Upsilon_n\in\{\brrho,
\brrho'\}$. Hence

    $$\eqref{ApproxRho'}
    = \int_{\bar{x}}^x \sum_{finite}\sum_{\widetilde{i}_1,\dots,\widetilde{i}_k}\mathfrak{D}^{\widetilde{i}_1,\dots,\widetilde{i}_k}_{m_1+1,\dots,\widetilde{m}_{\widetilde{k}}}(\widetilde{h}_1,\dots,\widetilde{h}_{n,k}\Upsilon_n)(s)ds$$

   Decompose the last integral as

    $$\int_{\bar{x}}^x \sum_{finite}\sum_{\widetilde{i}_1,\dots,\widetilde{i}_k}\mathfrak{D}^{\widetilde{i}_1,\dots,\widetilde{i}_k}_{m_1+1,\dots,\widetilde{m}_{\widetilde{k}}}(\widetilde{h}_1,\dots,\widetilde{h}_{k}\Upsilon_n)(s)ds
    = \sum_{finite}\sum_{\widetilde{i}_1,\dots,\widetilde{i}_k} \mathfrak{D}(\widetilde{h}_1,\dots,\widetilde{h}_{k}\Upsilon_n)(\bar{x})(x-\bar{x})+$$ $$\hspace{2cm}+\int_{\bar{x}}^x \bigg[ \mathfrak{D}^{\widetilde{i}_1,\dots,\widetilde{i}_k}_{m_1+1,\dots,\widetilde{m}_{\widetilde{k}}}(\widetilde{h}_1,\dots,\widetilde{h}_{k}\Upsilon_n)(s)
    \;-\;\mathfrak{D}^{\widetilde{i}_1,\dots,\widetilde{i}_k}_{m_1+1,\dots,\widetilde{m}_{\widetilde{k}}}(\widetilde{h}_1,\dots,\widetilde{h}_{k}\Upsilon_n)(\bar{x})\bigg]ds$$

We now invoke Proposition \ref{derivativeforD} again which together with \eqref{NonCritInt} implies
that
$\mathfrak{D}^{\widetilde{i}_1,\dots,\widetilde{i}_k}_{m_1+1,\dots,m_k}
(\widetilde{h}_1,\dots,\widetilde{h}_{k},\Upsilon_n)$
is differentiable on $[x;\bar{x}].$ Moreover, by Proposition \ref{PrDecDh} its derivative
is bounded by a constant $M$. Hence the last integrand in the above formula is $O(\epsilon)$
and so the integral is $O(\epsilon^2).$ Accordingly

$$\eqref{ApproxRho'}
    =(x-\bar{x})\sum_{finite}\sum_{\widetilde{i}_1,\dots,\widetilde{i}_k} \mathfrak{D}(\widetilde{h}_1,\dots,\widetilde{h}_{k}\Upsilon_n)(\bar{x})+O(\eps^2).$$

Hence
    $$\lim_{x\to \bar{x} } \frac{\rho_k(x)-\rho_k(\bar{x})}{x-\bar{x}}
    =\lim_{x\to \bar{x}}\sum_{finite}\sum_{1\leq i_1,\dots,i_k<n} \sum_{finite}\sum_{\widetilde{i}_1,\dots,\widetilde{i}_k} \mathfrak{D}(\widetilde{h}_1,\dots,\widetilde{h}_{k}\Upsilon_n)(\bar{x})
    +O\left(\epsilon+\frac{\lambda^{-n}+\theta^n}{\eps}\right). $$

As $x$ approaches $\bar{x}$, $n$ goes to $\infty$, hence $\Upsilon_n$ converges to $\rho$ or $\rho_1$.
Thus,

$$\lim_{x\to \bar{x} } \frac{\rho_k(x)-\rho_k(\bar{x})}{x-\bar{x}} = \sum_{finite}\sum_{i_1,\dots,i_k=1}^{\infty} \sum_{finite}\sum_{\widetilde{i}_1,\dots,\widetilde{i}_k} \mathfrak{D}(\widetilde{h}_1,\dots,\widetilde{h}_{k}, \widetilde{\rho})(\bar{x})
=\rho_{k+1}(\bar{x}). \qedhere $$
\end{proof}

In particular, we have the following result which also follows from \cite{Szewc}.

\begin{corollary}
    If $c$ is periodic of period $p$, then $\rho$ differentiable except for a finite set of points.
\end{corollary}
\begin{proof}
    If $\bar{x}$ does not belong to the orbit of $c$ (which is a finite set) then we can pick any
    $\beta>\max(\theta, 1/\lambda)$ and pick $j_0\geq 1$ large enough so that $d(\bar{x},\bar{c})\geq \beta^j$ for all $j \geq j_0$, where $\bar{c}=\max\{c_1,c_2,\dots,c_p \}.$
\end{proof}

\subsection{Nondifferentiability set.}\label{Nondifferentiability}

As we saw in Proposition \ref{higherderivative}, if the critical orbit does not approach a point $x$ exponentially fast, then the density function $\rho$ is differentiable at $x$.
In this subsection, we obtain a partial converse to this statement that is, if the critical point does approach exponentially fast with sufficiently high exponent then we cannot have differentiability.

\begin{definition}
For $\beta<1$, define
$$\mathcal{N}_{\beta} = \{\bar{x}\;:\; d(c_n,\bar{x})\leq \beta^n \; \mbox{for infinitely many }n'\mbox{s}   \}.$$
\end{definition}

\begin{proposition}
$\mathcal{HD}(\mathcal{N}_{\beta})=0$ where $\mathcal{HD}$ denotes the Hausdorff dimension.
\end{proposition}

\begin{proof}

Define $U_n$ as the ball centered at $c_n$ of radius $\beta^n$.
Given $\epsilon>0$ let $n_0\geq 1$ such that $\beta^{n_0} \leq \epsilon$.
Then, $\{ U_n \}_{n\geq n_0}$ is an $\epsilon-$cover of $\mathcal{N}_{\beta}$.

Note that $|U_n|=2\beta^n$.$\;\;$Hence, for any $s\geq0$ we have that

$$
\mathcal{H}_{\epsilon}^s(\mathcal{N}_{\beta} ) \leq \sum_{n \geq n_0} |U_n|^s
\leq \sum_{n \geq n_0} |U_n|^s
=\frac{2\beta^{n_0 s}}{1-\beta^s}
< \infty. $$
Therefore $\mathcal{HD}(\mathcal{N}_{\beta})=0$.
\end{proof}

\begin{proposition}\label{cndense}
If $\{c_n\}$ is dense in some interval $I\subset[0,1]$ then $\mathcal{N}_{\beta}$ is uncountable for all $\beta<1$.
\end{proposition}

We have already mentioned in Remark \ref{RmNvsB} the closure of $\{c_n\}$ contains an interval for
a typical PEUM.

\begin{proof}
Define $L_n=[c_n-\beta^n,c_n+\beta^n]$.

Since $\{c_n\}$ is dense, there exists $c_{n_1}$ such that $L_{n_1}$ is strictly contained in $I.\;\;$Set $M_1=L_{n_1}$.

Now, again using the density of $\{c_n\}$, there exist $c_{n_{(1,1)}} \in (c_{n_1}-\beta^{n_1},c_{n_1})$ and $c_{n_{(1,2)}} \in (c_{n_1},c_{n_1}+\beta^{n_1})$ such that $L_{n_{(1,1)}}$ and  $L_{n_{(1,2)}}$ are strictly contained in $(c_{n_1}-\beta^{n_1},c_{n_1})$ and  $(c_{n_1}-\beta^{n_1},c_{n_1})$ respectively.  $\;\;$ Set $M_2= L_{n_{(1,1)}} \cup L_{n_{(1,2)}}$.$\;\;$

Continuing this procedure we inductively define $M_n$ and set $M=\displaystyle \bigcap_{n\geq 1} M_n.$
$M$ is a Cantor set which is contained in $\mathcal{N}_\beta.$
Since $M$ is uncountable, so is $\mathcal{N}_{\beta}$.
\end{proof}

\begin{lemma} \label{LmNonDDensity}
    If
\begin{equation}
\label{BetaMaxDer}
\beta(\max_{x}  |f'(x)|)<1
\end{equation}

 and $\bar{x} \in \mathcal{N}_{\beta}$ then $\rho$ is non-differentiable at $\bar{x}$
\end{lemma}

\begin{proof}
    Suppose $\rho$ is differentiable at $\bar{x}$.
    Since $\bar{x} \in \mathcal{N}_{\beta}$, there exists a sequence $n_j$
    $d(\bar{x},c_{n_j}) \leq \beta^{n_j}.$
    Without loss of generality, assume $\bar{x} < c_{n_j}.$

    Let $y_1$ and $y_2$ be two arbitrary points  such that

    $$\bar{x}<y_1<c_{n_j}<y_2<c_{n_j}+\beta^{n_j}.$$

    Since $\rho$ is assumed to be differentiable at $\bar{x}$,
    we have that
$|\rho(y_i)-\rho(\bar{x})|\leq M\beta^{n_j}$ for $i=1,2$
and hence
$$|\rho(y_1)-\rho(y_2)| \leq 2M\beta^{n_j}.$$
Accordingly
    $$\frac{\rho{c}}{(\max|f'|)^{n_j}}\leq |\alpha_{n_j}|
    =\lim_{y_1\uparrow c_{n_j}, y_2\downarrow c_{n_j}}|\rho(y_2)-\rho(y_1)|\leq 2M\beta^{n_j}$$
    where the first inequality follows from  Lemma \ref{LmSaltRho}.
For large $j$ this inequality in incompatible with \eqref{BetaMaxDer}. Hence
$\rho$ can not be differentiable at $\bar{x}$.
\end{proof}

\subsection{Whitney smoothness}\label{Whitney}
\begin{proof}[Proof of Theorem \ref{ThW}, part (C)]

The case $k=1$ follows from Theorem \ref{derivativeforD}.

Let $k \geq 2$ and pick $1>\beta > \max\{\lambda^{-\frac{n}{k}},\theta^{\frac{n}{k}} \}$.$\;$Let $\bar{x} \notin \mathcal{N}_{\beta}$, let $\epsilon>0$ be very small.

Once again, let $n$ be the maximal number such that $c_j \notin [x;\bar{x}]$ for all $j\leq n.$

Then, similar to the proof of Theorem \ref{higherderivative},
\begin{equation}\label{epsilongreaterthanlambdaandomega}
\epsilon^k>\lambda^{-n}\;\mbox{ and }\; \epsilon^k > \theta^n
\end{equation}

Since $\rho=\mathcal{L}^n(1) + O(\theta^n)$, for $0 \leq s \leq k-1$, Proposition \ref{derivativeforD} implies
\begin{equation*}
\rho_{s} = \sum_{finite} \mathop{\sum_{k\leq i_1,\dots,i_k<n}}_{i_1 \geq 1 ,\dots,i_k \geq 1} \mathcal{D}_{m_1,\dots,m_j}^{i_1,\dots,i_j}(h_{1,s},\dots,h_{j-1,s},\mathcal{L}^n(1))+O(\lambda^{-n}+\theta^n).  \end{equation*}

To simplify the notation, let
\begin{equation*}
\rho_{s,n} = \sum_{finite} \mathop{\sum_{k\leq i_1,\dots,i_k<n}}_{i_1 \geq 1 ,\dots,i_k \geq 1} \mathcal{D}_{m_1,\dots,m_j}^{i_1,\dots,i_j}(h_{1,s},\dots,h_{j-1,s},\mathcal{L}^n(1)).
\end{equation*}

By definition of $n$ and since $f \in C^{k+2}$, $\rho_{k-1,n}$ is $C^2$ in $B(\bar{x},\epsilon)=\{y\;:\;|y-\bar{x}|< \epsilon \}$.$\;$Hence, if $x \in B(\bar{x},\epsilon)$,
\begin{eqnarray*}
  \rho_{k-1}(x)-\rho_{k-1}(\bar{x})&=& \rho_{k-1,n}(x) - \rho_{k-1,n}(\bar{x})+O(\lambda^{-n}+\theta^n)\\
  &=& \int_{\bar{x}}^x \rho_{k-1,n}'(y)\;dy+O(\lambda^{-n}+\theta^n)\\
  &=& \int_{\bar{x}}^x \rho_{k-1,n}'(y)- \rho_{k-1,n}'(\bar{x})dy+ \rho_{k-1,n}(\bar{x})(x-\bar{x})+O(\lambda^{-n}+\theta^n)\\
  &=&O(\epsilon^2)+\rho_{k-1,n}'(\bar{x})(x-\bar{x})+O(\lambda^{-n}+\theta^n). \\
\end{eqnarray*}

By Proposition \ref{derivativeforD}, $\rho'_{k-1,n}(\bar{x})=\rho_k(\bar{x})+O(\lambda^{-n}+\theta^n).$ Thus
\begin{equation*}
\rho_{k-1}(x)-\rho_{k-1}(\bar{x})=\rho_k(\bar{x})(x-\bar{x})+O(\epsilon(\lambda^{-n}+\theta^n))+O(\lambda^{-n}+\theta^n)+O(\epsilon^2)  .\end{equation*}

Inequalities \eqref{epsilongreaterthanlambdaandomega} imply that
\begin{equation}\label{rhox-rhobarx}
\rho_{k-1}(x)-\rho_{k-1}(\bar{x})=\rho_k(\bar{x})(x-\bar{x})+O(\epsilon^{k+1})+O(\epsilon^k)+O(\epsilon^2)
=\rho_k(\bar{x})(x-\bar{x})+O(\epsilon^2).
\end{equation}

Now, note that if $x \in B(\bar{x},\epsilon)$, then
\begin{equation}
\label{DeltaRhok-2}
  \rho_{k-2}(x)-\rho_{k-2}(\bar{x})=  \rho_{k-2,n}(x)-\rho_{k-2,n}(\bar{x})+O(\lambda^{-n}+\theta^n)
\end{equation}
$$  = \int_{\bar{x}}^x \rho_{k-2,n}'(y)\;dy+O(\lambda^{-n}+\theta^n). $$

By Proposition \ref{derivativeforD}, $\rho_{k-2,n}'(y)=\rho_{k-1}(y)+O(\lambda^{-n}+\theta^n)$.
Combining \eqref{DeltaRhok-2} with \eqref{rhox-rhobarx} and using that
$\epsilon^{k+1}<\epsilon^k<\epsilon^3$ we get
\begin{eqnarray*}
  \rho_{k-2}(x)-\rho_{k-2}(\bar{x})&=&\int_{\bar{x}}^x\rho_{k-1}(y)\;dy+O(\epsilon(\lambda^{-n}+\theta^n) )+O(\lambda^{-n}+\theta^n)\\
  &=& \int_{\bar{x}}^x\rho_{k-1}(\bar{x})+\rho_k(\bar{x})(y-\bar{x})\;dy+O(\epsilon^3)+O(\epsilon(\lambda^{-n}+\theta^n) )+O(\lambda^{-n}+\theta^n)\\
  &=& \rho_{k-1}(\bar{x})(x-\bar{x})+\rho_k(\bar{x})\frac{(x-\bar{x})^2}{2}+O(\epsilon^3).\\
\end{eqnarray*}
Continuing this recursive argument we get
\begin{equation*}
  \rho_s(x)-\rho_s(\bar{x})=\bigg(\sum_{j=0}^{k-s-1}\rho_{k-j}(\bar{x})\frac{(x-\bar{x})^{k-s-j}}{(k-s-j)!}\bigg)+O(\epsilon^{k-s+1})
\end{equation*} for all $s=0,\dots,k-1$.
In particular, when $s=0$, we have the desired result.
\end{proof}

Parts (A) and (B) of Theorem \ref{ThW} follows from
Theorem \ref{higherderivative}, Proposition \ref{cndense}
and Lemma \ref{LmNonDDensity}. Since part (C) was just proven, the proof of Theorem \ref{ThW} is complete.

\end{document}